\documentclass[preprint]{imsart}

\RequirePackage{amsthm,amsmath,amsfonts,amssymb}
\RequirePackage[numbers]{natbib}
\RequirePackage[colorlinks,citecolor=blue,urlcolor=blue]{hyperref}
\RequirePackage{graphicx}

\usepackage{support-caption}
\usepackage{subcaption}
\usepackage{textcomp}

\startlocaldefs

\numberwithin{equation}{section}
\theoremstyle{plain}

\newtheorem{theorem}{Theorem}[section]
\newtheorem{proposition}[theorem]{Proposition}
\newtheorem{lemma}[theorem]{Lemma}
\theoremstyle{remark}
\newtheorem{definition}[theorem]{Definition}
\newtheorem{remark}[theorem]{Remark}
\newtheorem*{example}{Example}

\endlocaldefs

\begin{document}

\begin{frontmatter}
\title{Harmonic analysis meets stationarity: A general framework for series expansions of special Gaussian processes}
%\title{A sample article title with some additional note\thanksref{t1}}
\runtitle{Series expansions of special Gaussian processes}
%\thankstext{T1}{A sample additional note to the title.}

\begin{aug}
\author[A]{\fnms{Mohamed} \snm{Ndaoud}\ead[label=e1]{ndaoud@usc.edu}}
%%%%%%%%%%%%%%%%%%%%%%%%%%%%%%%%%%%%%%%%%%%%%%
%% Addresses                                %%
%%%%%%%%%%%%%%%%%%%%%%%%%%%%%%%%%%%%%%%%%%%%%%
\address[A]{Department of Mathematics,
University of Southern California, Los Angeles, USA.
\printead{e1}}

\end{aug}

%\editor{}

\maketitle

\begin{abstract}
In this paper, we present a new approach to derive series expansions for some Gaussian processes based on harmonic analysis of their covariance function. In particular, we propose a new simple rate-optimal series expansion for fractional Brownian motion. The convergence of the latter series holds in mean square and uniformly almost surely, with a rate-optimal decay of the remainder of the series. We also develop a general framework of convergent series expansions for certain classes of Gaussian processes with stationarity. Finally, an application to optimal functional quantization is described.
\end{abstract}

\end{frontmatter}

%% if your bibliography is in bibtex format, uncomment commands:
%\bibliographystyle{imsart-number} % Style BST file
%\bibliography{bibliography}       % Bibliography file (usually '*.bib')

%% or include bibliography directly:
\section{Introduction}
Let $B$ = $(B_{t})_{t\in\mathbb{R}^{+}}$ be a centered Gaussian process. $B$ is called fractional Brownian motion (fBm) with Hurst exponent $H \in \left(0,1\right)$ if it has the following covariance structure
\[ \forall t,s\in\mathbb{R}^{+},\quad \mathbf{E}\left(B_{s}B_{t}\right) = \frac{1}{2}\bigg( t^{2H} + s^{2H} - |t-s|^{2H}\bigg). \]
Fractional Brownian motion is a self-similar process i.e.  $\forall c,t>0, \; B_{ct}\overset{\textbf{d}}{=}c^{H}B_{t}$, with stationary increments i.e. $\forall s,t>0, \; B_{t}-B_{s}\overset{\textbf{d}}{=}B_{t-s}$, where $\overset{\textbf{d}}{=}$ denotes equality in distribution. When $H=1/2$, fractional Brownian motion coincides with the standard Brownian motion.
Sample paths of fBm are almost surely H\"{o}lder-continuous of any order strictly less than $H$, and hence are almost surely everywhere continuous. 

One of the main challenges with fBm is its simulation, as it is the case for Gaussian processes with a complex covariance structure in general. The circulant embedding method, described in \citep{embed}, is one of the most efficient algorithms to simulate either stationary Gaussian processes or Gaussian processes with stationary increments on a finite interval $[0,T]$ for some $T>0$. In particular, the latter algorithm has an $N\log N$ complexity, where $N$ is the number of time steps discretizing $[0,T]$. This complexity is to be compared with linear complexity for the standard Brownian motion due to the independence of its increments. Besides, circulant embedding does not allow local refinement.

Alternative approximation methods to simulate a Gaussian process involve its Karhunen-Lo\`{e}ve expansion. The latter expansion is explicitly known for some processes such as the Brownian motion, the Brownian bridge \citep{bridge} and the Ornstein-Uhlenbeck process \citep{OU}, to name a few. Unfortunately, this expansion is not explicit for fBm. In order words, we do not know explicit characterization of eigenvalues $(\nu_k)_k$ and eingevectors of the Karhunen-Lo\`{e}ve expansion of fBm. It is worth mentionning here that exact first order asymptotics of the latter eigenvalues are known (cf. \citep{bronski1,bronski2,brunski4,bronski3}), and more recently second order  asymptotics were derived (cf. \citep{chigansky2018exact}). We give below the first order asymptotics of $(\nu_k)_k$, for $H\in(0,1)$, that are given by
\begin{equation}\label{eq:asymp}
   \nu_k \sim \Gamma(2H+1) \sin(\pi H) (k\pi)^{-2H-1},
\end{equation}
where $\Gamma$ is the gamma function. Such sharp asymptotic results are useful to control precisely the error of approximation, when using series expansion to generate fBm, among other applications.
%One of the main difficulties of the fBm is its simulation. Because of the correlation between all its increments we are constrained to use the whole covariance structure to generate %one simple path eventually by using the Cholesky decomposition method. One way to handle this problemma is to use the Karhunen-Loeve expansion, unfortunately there is no explicit %formula for its basis as far as we know except for the case H=1/2 where
%\[ W_{t} = \sqrt{2}\sum_{n=1}^\infty \frac{\sin(n-\frac{1}{2})\pi t}{(n-\frac{1}{2})\pi}X_{n}, \; \; t\in[0,1] \]
%and ($X_{n})_{n\geq 1}$ are i.i.d standard Gaussian random variables.

{\bf Notation.}
In the rest of this paper we use the following notation. For given positive sequences $a_{n}$ and $b_n$, we say that $a_{n} = \mathcal{O}(b_{n})$ (resp $a_{n} = \Omega(b_n)$) when $a_{n} \leq c b_{n}$ (resp $a_{n} \geq c b_{n}$) for some absolute constant $c>0$ and $a_n \sim b_n$ when $a_n/b_n \to 1$. We write $a_{n} \asymp b_{n} $ when $a_{n} = \mathcal{O}(b_{n})$ and $a_{n}=\Omega(b_{n})$. For $X\in \mathbb{R}^{p}$, we denote by $\|X\|$ the Euclidean norm of $X$. For $x,y \in \mathbb{R}$, we denote by $x\vee y$ the maximum value between $x$ and $y$. In particular $x \vee 0$ will be denoted by $x_{+}$. $\Re(z)$ denotes the real part of complex variable $z$. Finally, we denote by $C[0,T]$ the space of continuous functions on $[0,T]$ endowed with the sup-norm.

\subsection{Related literature}
There are various approaches to approximate fBm such as \citep{leon2020stratonovich} and \citep{li2011approximations} to name a few. A complete overview of these approximations falls beyond the scope of this paper. This work focuses on the approximations based on series expansions. In \citep{ayache}, one of the first rate-optimal series expansion of fBm based on Wavelet series approximations is presented. For the sake of brevity, we only present , in what follows, trigonometric series approximations since they can be compared to the framework we expose further.  

The first trigonometric series expansion for fBm on $[0,1]$ was discovered in \citep{ZT}. For $H \in (0,1)$, the series $(B^{H}_{t})_{t\in[0,1]}$ is given by
\[ B^{H}_{t} = \sum_{n=1}^\infty \frac{\sin(x_{n}t)}{x_{n}}X_{n} + \sum_{n=1}^\infty \frac{1 - \cos( y_{n}t)}{y_{n}}Y_{n},\quad t\in[0,1], \]
where $(X_{n})_{n\geq 1}$ and $(Y_{n})_{n\geq 1}$ are centered independent Gaussian random variables, $(x_{n})_{n\geq 1}$ is the sequence of positive roots of the Bessel function $J_{-H}$, and $(y_{n})_{n\geq 1}$ the sequence of positive roots of the Bessel function $J_{1-H}$. The variance of the Gaussian variables involved in the series is given by 
$$ \forall n \geq 1, \quad Var(X_{n}) =  2c_{H}^{2}x_{n}^{-2H}J_{1-H}^{-2}(x_{n}) \quad \text{and} \quad Var(Y_{n}) =  2c_{H}^{2}y_{n}^{-2H}J_{-H}^{-2}(y_{n}),$$
where $c_{H}^{2}=\pi^{-1}\Gamma(1+2H)\sin(\pi H)$. Dzhaparidze and Van Zanten \citep{ZT} prove rate-optimality of the above series expansion in the following sense.
\begin{definition}
Let $H\in(0,1)$ and $B^H$ a fBm with Hurst exponent $H$ on $[0,T]$ for some $T>0$. Assume that $B^{H}$ is given by the series expansion
\[ \forall t \in [0,T], \quad B^H_{t} = \sum_{i=0}^{\infty}Z_{i}e_{i}(t),\]
where  $(Z_{i})_{i\in\mathbb{N}}$ is a sequence of centered independent Gaussian random variables and $(e_{i})_{i\in\mathbb{N}}$ a sequence of continuous deterministic functions. $B^{H}$ is said to be uniformly rate-optimal if
\[ \mathbf{E}\underset{t\in[0,T]}{\sup}\left|\sum_{i=N}^{\infty}Z_{i}e_{i}(t)\right|\asymp N^{-H}\sqrt{\log{N}}. \]
\end{definition}
In \citep{KH}, the rate $N^{-H}\sqrt{\log{N}}$ is shown to be optimal. In other words, rate-optimality means that no other series expansion of fBm can achieve a faster rate of convergence. We explain later how rate-optimality implies uniform convergence of the series, almost surely.

Another rate-optimal trigonometric series expansion for fBm, in the case $H\in(1/2,1)$, is derived in \citep{GL}, that is close to our representation. For $H\in(1/2,1)$, this expansion takes the form
\[  B_{t} = \sqrt{a_{0}}tX_{0} + \sum_{k=1}^\infty \sqrt{a_{k}} \left ( \sin(k\pi t) X_{k} + (1-\cos(k\pi t)) X_{-k}   \right), \quad t\in[0,1],  \]
where
\begin{equation}\label{eq:inglot1}
 a_{0} = \frac{\Gamma(2-2H)}{B(H-\frac{1}{2},\frac{3}{2}-H)(2H-1)} ,  
 \end{equation}
 \begin{equation}\label{eq:inglot2}
 \forall k \geq 1 \, \quad a_{k} = \frac{\Gamma(2-2H)}{B(H-\frac{1}{2},\frac{3}{2}-H)}2\Re(i\exp^{-i\pi H}\gamma(2H-1,ik\pi))(k\pi)^{-2H-1} ,  
\end{equation}
and $(X_{k})_{k\in\mathbb{Z}}$ is a sequence of independent standard Gaussian random variables. The functions $\Gamma$, $B$, and $\gamma$ are the Gamma, Beta, and complementary (lower) incomplete Gamma functions,
respectively. Even if this representation is easier to evaluate than the previous one, it still requires computation of special functions and only holds for $H\in(1/2,1)$.
\subsection{Main contribution}
In this paper, we give a constructive representation of fBm for all $H\in(0,1)$ which is only based on harmonic analysis of its covariance function. Our approach is inspired from the Karhunen-Lo\`{e}ve expansion. The latter expansion is obtained through
an interesting application of the spectral theorem for compact normal operators, in conjunction
with Mercer\textquotesingle s theorem. We give here a sketch of its proof for the reader's convenience. Let $K_{B}(.,.)$ be the covariance function of some process $B$ of interest on $[0,1]^{2}$. Mercer\textquotesingle s theorem is a series representation of $K_{B}$ 
based on the diagonalization of the following linear operator
$$\begin{array}{ccccc}
T_{K_{B}} & : & L^{2}\left[0,1\right] & \to & L^{2}\left[0,1\right] \\
 & & f & \to & \int_{0}^{1}K_{B}\left(s,.\right)f(s)\mathrm{d}s,\\
\end{array}$$
where $L^{2}[0,1]$ is the space of square-integrable real-valued functions on $[0,1]$.
In particular, it states that there is an orthonormal basis $(e_{i})_{i\in \mathbb{N}}$ of $L^{2}[0, 1]$ consisting of eigenfunctions of $T_{K_{B}}$ such that the corresponding sequence of eigenvalues $(\lambda_{i})_{i\in\mathbb{N}}$ is nonnegative. Moreover $K_{B}$ has the following representation
$$ \forall s,t \in [0,1], \quad K_B(s,t) = \sum_{i=0}^{\infty} \lambda_{i}e_{i}(t)e_{i}(s),$$
where the series converges in $L^{2}$.
Considering $(Z_{i})_{i \in \mathbb{N}}$ a sequence of independent standard Gaussian random variables, and assuming the uniform convergence of the following series
$$ \forall t \in [0,1],\quad X_{t} = \sum_{i=0}^{\infty}\sqrt{\lambda_{i}}Z_{i}e_{i}(t), $$
one may observe that $X$ is a centered Gaussian process on $[0,1]$ with a covariance function equal to $K_{B}$. Hence $X$ and $B$ have the same distribution on $[0,1]$. Since the corresponding eigenfunction sequence $(e_{i})_{i\in\mathbb{N}}$ is not explicit for fBm, we follow an alternative approach. 

In the case where $B$ is a stationary process or has stationary increments, $T_{K_{B}}$ becomes similar to a convolution operator. It is well known that the Fourier basis is a basis of eigenfunctions for the convolution operator, when the convolution kernel is periodic. In general, we may extend the kernel to an even periodic kernel. Since this modification applies to the covariance function $K_{B}$, there is no guarantee that the new symmetric function $\tilde{K}_{B}(.,.)$ is positive. In particular, the corresponding eigenvalues are not necessary positive. The last condition is crucial in our approach, since we need to take square root of the Fourier coefficients, as described in the Karhunen-Lo\`{e}ve proof. 

One of our main contributions, is to exhibit a new class $\Gamma$ of functions such that the Fourier coefficients are not only negative, but also lead to rate-optimal series expansions of Gaussian processes with stationarity. Let $T>0$ and  $\gamma$ be a real valued function on $(0,T]$. We say that $\gamma$ satisfies property ($\star$) if
\begin{itemize}
        \item $\gamma$  is continuously differentiable, increasing and concave. 
        \item $x^{\delta}\gamma'(x) = \mathcal{O}\left( 1 \right)$ as $x\to0^{+}$, for some $\delta\in[0,2)$.
\end{itemize}
We denote by $\Gamma$ the class of such functions. This class turns out to be large enough to include many examples of interest as we explain in Section \ref{sec:gen}.
As a consequence, we derive a new series expansion for fBm given by
\[  B^{H}_{t} = \sqrt{b_0}tX_{0} +  \sum_{k=1}^\infty \sqrt{b_k} \left( \sin(k\pi t) X_{k} +  \left(1-\cos(k\pi t)\right) X_{-k} \right), \quad t\in[0,1],  \]
where
\begin{equation}\label{eq:simo1}
\begin{cases}&  b_{0} := 0,\quad H<1/2 \\ & b_{0} := H, \quad H>1/2,
\end{cases}
\end{equation}
\begin{equation}\label{eq:simo2}
\forall k \geq 1, \quad \begin{cases}& b_{k} := -\int_0^1 t^{2H}\cos(k\pi t)\mathrm{d}t ,\quad H<1/2 \\ & b_{k} := \frac{2H(2H-1)}{(k\pi)^{2}}\int_0^1 t^{2H-2}\cos(k\pi t)\mathrm{d}t, \quad H>1/2,
\end{cases}
\end{equation}
and $(X_{k})_{k \in \mathbb{Z}}$ is a sequence of independent standard Gaussian random variables. This series is rate-optimal and its convergence holds uniformly almost surely. More generally, we also derive series expansion for a general class of Gaussian processes with covariance operator linked with the class $\Gamma$. Let us emphasize here that our series expansion in different from the one ine \citep{GL}. Indeed, it is enough to compare the first terms of the two series expansions to be convinced. We show, in Lemma \ref{lem:compare}, that for $H \in (1/2,1)$ we get that
\[
\frac{b_0}{a_0} = \Gamma(2H+1) \sin(\pi H).
\]
More interestingly, we also show in Lemma \ref{lem:compare} that, for $H\in(0,1)$, we have
\[
b_k \sim \Gamma(2H+1) \sin(\pi H) (k\pi)^{-2H-1}.
\]
This in particular implies that $b_k \sim \nu_k$ where $\nu_k$ was defined in \eqref{eq:asymp} as the exact first order asymptotic of eigenvalues of the Karhunen-Lo\`{e}ve decomposition for fBm. Hence our series expansion is not only rate-optimal but is asymptotically equivalent to the Karhunen-Lo\`{e}ve decomposition in terms of the squared mean error. This also implies that the error of quantization behaves asymptotically as if the Karhunen-Lo\`eve expansion were used, which is another advantage of our expansion.

Section \ref{sec:har} is devoted to some preliminaries. In Section \ref{sec:main}, we present our series expansion for fBm, where we prove both uniform convergence and rate-optimality. Next, we generalize this series expansion to a large class of Gaussian processes, before applying it to functional quantization. All proofs are deferred to the Appendix.

\section{ Preliminaries }\label{sec:har}
We start this section by presenting general harmonic properties of the class $\Gamma$. Let $T>0$ and $\gamma \in \Gamma$. Consider the corresponding Fourier sequence
\begin{equation}\label{eq:defck}
\forall k \in \mathbb{N}, \quad c_{k}(\gamma) := \frac{2}{T}\int_0^T \gamma(t)\cos{\frac{k\pi t}{T}}\mathrm{d}t.
\end{equation}
The next Proposition states some important properties of $c(\gamma) = (c_{k}(\gamma))_{k \in \mathbb{N}}$ when $\gamma$ satisfies ($\star$). In what follows, we write for brevity $c_{k}$ = $c_{k}(\gamma)$, as long as there is no ambiguity.
\begin{proposition}\label{pr:1}
Let $\gamma$ be a function satisfying ($\star$) and $c(\gamma)$ the sequence defined in \eqref{eq:defck}, then
\begin{itemize}
\item $c(\gamma)$ is well defined.
\item $\forall k \geq 1, \;$ $c_{k} \leq 0$.
\item $|c_{k}| = \mathcal{O}\left( \frac{1}{k^{2-\delta}} \right)$.
\end{itemize}
\end{proposition}
 It is usually not easy to reveal the sign of the Fourier coefficients of a given function. Nonetheless, this question can be answered for convex (or equivalently concave) functions. It turns out that these coefficients are positive for convex functions as claimed in \citep{wang1985convex}. For the class of functions satisfying $(\star)$, not only we can reveal the sign of these coefficients but also their asymptotic behaviour as shown in Proposition \ref{pr:1}. %The more general question of characterizing the class of functions with negative Fourier coefficients is beyond the scope of this paper. 
 It is also interesting to notice that for $\gamma \in \Gamma$, the singularity around $0^{+}$, captures the asymptotic behaviour of $c(\gamma)$. A thing that does not hold in general.

Inspection of the proof of Proposition \ref{pr:1}, shows that $\gamma$ has a finite limit at $0^{+}$, whenever $\delta \in [0,1)$.We will use in that case the notation $\gamma(0):=\underset{ x\to 0^{+} }{\lim} \gamma(x)$. The next lemma gives a useful Fourier expansion for functions in $\Gamma$.
\begin{lemma}\label{lem:1}
Let $\gamma$ be a function satisfying $(\star)$ for some $\delta \in [0,1)$. Then
\[ \forall t \in [-T,T], \quad \gamma(|t|) = \gamma(0) + \sum_{k=1}^\infty c_{k}\bigg(\cos\frac{k\pi t}{T}-1\bigg), \]
where convergence holds uniformly.
\end{lemma}
We present now a result that characterizes the rate of convergence of given series expansions. Let $\lambda:=(\lambda_{k})_{k\in\mathbb{N}}$ be a sequence of real numbers and $e:=(e_{k})_{k\in\mathbb{N}}$ a family of uniformly bounded and continuous functions on $[0,T]$. We say that $(\lambda,e)$ satisfies ($\star \star$) if 
\begin{itemize}
\item $\exists H>0$, such that $|\lambda_{k}| = \mathcal{O}\left(\frac{1}{k^{H+1/2}}\right)$.
\item  $\forall k \in \mathbb{N}$, $e_k$ is $L$-Lipschitz for some $L>0$ or equivalently
$$ \forall s,t \in[0,T], \quad \left|e_{k}(t)-e_{k}(s)\right| \leq L\left|t-s\right|.$$
\end{itemize}
Notice that for $\gamma \in \Gamma$, and setting $(e_k)_{k\geq 0}$ such that $e_k(x) = f(k\pi x)/k$ for $f$ one of the following functions $\{\cos(.),\sin(.),1-\cos(.)\}$, it is easy to check that $(\sqrt{-c(\gamma)},e)$ satisfies $(\star \star)$ for $\delta \in [0,1)$. 
\begin{theorem}\label{thm:1}
Let $(\lambda_{k})_{k\in\mathbb{N}}$ be a sequence of real numbers, $(Z_{k})_{k\in\mathbb{N}}$ a sequence of centered standard Gaussian random variables, and $(e_{k})_{k\in\mathbb{N}}$ a family of continuous functions on $[0,T]$. Assume that $(\lambda,e)$ satisfies ($\star \star$) for some $H>0$. Then the series $\sum_{k=0}^{N}\lambda_{k}e_{k}\left(\frac{k\pi .}{T}\right)Z_{k}$ converges almost surely in $C[0,T]$. Moreover, we have
$$  \mathbf{E} \operatorname*{sup}_{t\in[0,T]}\left|\sum_{k=N}^{\infty}\lambda_{k}e_{k}\left(\frac{k\pi t}{T}\right)Z_{k} \right| = \mathcal{O} \left( N^{-H}\sqrt{\log{N}}\right).
$$
\end{theorem}
 In what follows, we will repeatedly use Proposition \ref{pr:1} along with Theorem \ref{thm:1}. In fact, Proposition \ref{pr:1} describes the asymptotic behaviour of $c(\gamma)$ for $\gamma \in \Gamma$, while Theorem \ref{thm:1} characterizes the rate of convergence of given series expansions based on the asymptotic behaviour of $c(\gamma)$. 

\section{Constructing the fractional Brownian motion}\label{sec:main}

In this section, we present our series expansion for fBm and prove its convergence. The construction is based on harmonic decomposition of the auto-covariance function $\gamma$ on $[0,T]$ such that $\gamma(t) = |t|^{2H}$ for some $H\in(0,1)$. As described in the Introduction, diagonalization of the operator $T_{K}$ is not explicit for fbm. In order to benefit from the diagonalization of the convolution operator, we need to extend the auto-covariance function to a periodic function. The resulting function $\tilde{K}(.,.)$ is not guaranteed to be a covariance function. Luckily, harmonic properties of the class $\Gamma$ will be useful to get around this drawback. 

Since our approach does not hold for both cases, we give results separately for both fBm with $H\in(0,1/2)$ and $H\in(1/2,1)$, assuming that the series converge. We prove later the convergence and rate-optimality of these series.
\subsection{The series expansion}
The following theorem gives an explicit series expansion for fBm for $H\in(0,1/2)$, assuming convergence of the series.
\begin{theorem}\label{proof:thm:1}
Let H $\in \left(0,1/2\right)$. Consider the function $\gamma$ given by $ \gamma(t)=t^{2H}$, $\forall t \in [0,T]$ and denote by $c(\gamma)$ the sequence of its Fourier coefficients. Let $B$ be a stochastic process given by the series expansion
\[  \forall t \in [0,T], \quad  B_{t} = \sum_{k=1}^\infty \sqrt{-\frac{c_{k}}{2}} \left( Z_{k} \sin\frac{k\pi t}{T}  +Z_{-k} \left(1-\cos\frac{k\pi t}{T}\right)    \right), \]
where $(Z_{k})_{k\in\mathbb{Z}}$ denotes a sequence of independent standard Gaussian random variables. Then
\[\forall (s,t) \in [0,T]^{2}, \;  \mathbf{E}\left(B_{s}B_{t}\right) = \frac{1}{2}\left( t^{2H} + s^{2H} - |t-s|^{2H}\right).\]
\end{theorem}
In particular, assuming its convergence, the stochastic process $(B_t)_{t\in[0,T]}$ defined in Theorem \ref{proof:thm:1} is Gaussian and has fbm's covariance which implies that it is a fBm on $[0,T]$. The proof of Theorem \ref{proof:thm:1} does not hold for the case $H\in(1/2,1)$. In fact, for $H\in(1/2,1)$, the Fourier coefficients $c(\gamma)$ have alternating signs and we cannot consider the square root of these coefficients anymore. This is basically due to the fact that $\gamma$ does not satisfy property $(\star)$ in this case, since $\gamma$ is convex and increasing. %In particular, the change in the sign of $c(\gamma)$ is partially due to the smoothness of $\gamma'$ around $0$. 
One may still notice that $\gamma''$ satisfies $(\star)$. The next lemma, gives a link between $c(\gamma)$ and $c(\gamma'')$. 
\begin{lemma}\label{lem:2}
Let $\gamma$ be a twice differentiable function on $(0,T)$ such that $ \gamma'(0) \neq \gamma'(T)$. Define $f$ such that
$$\forall t \in[0,T], \quad f(t) =  \gamma(t) - \frac{\gamma'(T)-\gamma'(0)}{2T}\left(t+\frac{T\gamma'(0)}{\gamma'(T)-\gamma'(0)}\right)^{2},$$
then $\forall k \geq 1$ we have
$$  c_{k}(f) = \left(\frac{T}{k\pi}\right)^{2}c_{k}(-\gamma'').$$
\end{lemma}
Armed with the above lemma, we can not state our next result. The following theorem gives an explicit series expansion for fBm when $H \in (1/2,1)$, again assuming the series convergence.
\begin{theorem}\label{pr:thm:2}
Let H$\in\left(1/2,1\right)$. Consider the function $\gamma$ given by $ \gamma(t)=-2H(2H-1)t^{2H-2}$, $\forall t \in [0,T]$ and denote by $c(\gamma)$ the sequence of its Fourier coefficients. Let $B$ be a stochastic process given by the series expansion
\[\forall t \in [0,T], \quad B_{t} =\sqrt{HT^{2H-2}}tZ_{0} + \sum_{k=1}^{\infty} \frac{T}{k\pi} \sqrt{-\frac{c_{k}}{2}}\left(\sin{\frac{k\pi t}{T}}Z_{k} + \left(1-\cos{\frac{k\pi t}{T}} \right)Z_{-k}   \right), \]
where $(Z_{k})_{k\in\mathbb{Z}}$ denotes a sequence of independent standard Gaussian random variables. Then
\[ \forall (s,t) \in [0,T]^{2}, \;  \mathbf{E}\left(B_{s}B_{t}\right) = \frac{1}{2}\left( t^{2H} + s^{2H} - |t-s|^{2H}\right). \]
\end{theorem}
Similarly, the above stochastic process is a fBm on $[0,T]$. In the following Lemma, we show that the series expansion in Theorem \ref{pr:thm:2} is different from the one ine \citep{GL}. Indeed, it is enough to compare the first terms of the series expansions to be convinced.
\begin{lemma}\label{lem:compare} Let $(a_k)_k$ and $(b_k))_k$ the sequences defined in \eqref{eq:inglot1}-\eqref{eq:inglot2} and \eqref{eq:simo1}-\eqref{eq:simo2} respectively. Then for $H\in (1/2,1)$ we have
\[
\frac{b_0}{a_0} = \Gamma(2H+1) \sin(\pi H).
\]
Moreover, for $H\in (0,1)$, we have
\[
b_k \sim \Gamma(2H+1) \sin(\pi H) (k\pi)^{-2H-1}.
\]
\end{lemma}
This in particular implies that $b_k \sim \lambda_k$ where $\lambda_k$ was defined in \eqref{eq:asymp} as the exact first order asymptotic of eigenvalues of the Karhunen-Lo\`{e}ve decomposition for fBm. Hence our series expansion is not only rate-optimal but is asymptotically equivalent to the Karhunen-Lo\`{e}ve decomposition in terms of squared mean error.
\subsection{Convergence and rate-optimality}
After presenting a new explicit representation of fBm, we prove now its convergence in both mean square error and almost surely. We also show its uniform rate-optimality. For the rest of the section, we denote more precisely by $(c_{k})_{k \geq 0}$ the following sequence
\begin{equation}\label{eq:defffff}
\begin{cases}& c_{0} := 0,\quad \quad\quad\quad H<1/2, \\ & c_{0} := HT^{2H-2}, \quad 1/2<H,
\end{cases}
\end{equation}
and for $k\geq 1$
\begin{equation}\label{eq:deftk2}
\begin{cases}& c_{k} := \frac{2}{T}\int_0^T t^{2H}\cos\frac{k\pi t}{T}\mathrm{d}t ,\quad \quad\quad \quad\quad\quad  H<1/2, \\ & c_{k} := -\frac{4H(2H-1)T}{(k\pi)^{2}}\int_0^T t^{2H-2}\cos\frac{k\pi t}{T}\mathrm{d}t, \quad 1/2<H.
\end{cases}
\end{equation}
One may first notice, based on Proposition \ref{pr:1}, that $|c_{k}|= \mathcal{O}\left(\frac{1}{k^{2H+1}}\right)$. We consider the series expansion constructed in the previous section and that is given by
\begin{equation}\label{eq:deff}
    \forall t \in [0,T], \;  B_{t} = \sqrt{c_{0}}tZ_{0} +  \sum_{k=1}^\infty \sqrt{-\frac{c_{k}}{2}} \left (  Z_{k}\sin\frac{k\pi t}{T} + Z_{-k} \left(1-\cos\frac{k\pi t}{T}\right)    \right).
\end{equation}
The next two theorems show the convergence of the above series in mean square error and almost surely uniformly.
Let $N \geq 1 $ be level of truncation. In what follows, we denote by $B^{N}$ the truncated series of $B$ and that is given by
\begin{equation}\label{eq:tronc}
    B_{t}^{N} = \sqrt{c_{0}}tZ_{0} + \sum_{k=1}^{N} \sqrt{-\frac{c_{k}}{2}} \left( Z_{k}\sin\frac{k\pi t}{T}  + Z_{-k}  \left(1-\cos\frac{k\pi t}{T}\right)   \right).  
\end{equation}
\begin{proposition}\label{prop:11}
Let $B_{t}$ and $B_{t}^{N}$ be defined by \eqref{eq:defffff}-\eqref{eq:tronc} and $H\in(0,1)$. Then,
\[ \operatorname*{sup}_{t\in[0,T]} \sqrt{\mathbf{E}\big(B_{t}-B_{t}^{N} \big)^2}= \mathcal{O}\big(N^{-H}\big).\]
\end{proposition}
Proposition \ref{prop:11} shows the mean square convergence of $B^{N}$. Since $B^{N}$ is a centered Gaussian process, and using the fact that the Gaussian Hilbert space is complete, we deduce that $B$ is a centered Gaussian process with the same covariance as fBm. It follows that $B$ is a fractional Brownian motion on $[0,T]$. We turn now to the question of rate-optimality of the series expansion. 
\begin{proposition}\label{thm:opt}
Let $B$ be the series expansion defined in \eqref{eq:deff}. Almost surely, $B_{t}^{N}$ converges uniformly, and its rate of convergence is given by
\[  \mathbf{E} \operatorname*{sup}_{t\in[0,T]}\big|B_{t}-B_{t}^{N} \big|  \operatorname*{\asymp}  N^{-H}\sqrt{\log(N)} .\]
\end{proposition}
Proposition \ref{thm:opt} implies, in particular, that the new series expansion, we have derived for fBm, is rate-optimal.

\section{Generalization to special Gaussian processes}\label{sec:gen}
In this section, we develop a general framework for series expansions of special Gaussian processes. For these processes, we prove almost sure uniform convergence and give the corresponding rate of convergence. The question of rate-optimality of the presented series expansions is left for further research. We refer the reader to Proposition 4 in \citep{PG} that gives some hints on rate-optimality of our constructions. 

The next theorem generalizes the case of fBm. We derive a series expansion for a class of Gaussian processes with stationary increments. Let $\gamma$ be a function satisfying $(\star)$ for some $\delta \in [0,1)$, and let $X$ be a centered Gaussian process. We say that $X$ is a Gaussian process of type $(A)$, if it is characterized by the following covariance structure
$$
\forall t,s \in [0,T],\quad \mathbf{E}\left(X_{t}X_{s}\right) = \frac{1}{2}\left(\gamma(t) + \gamma(s) - \gamma(|t-s|)\right).
$$
\begin{theorem}\label{theorem1}
Assume that $\gamma$ satisfies $(\star)$ for some $\delta \in [0,1)$, and let $c(\gamma)$ be the sequence of its Fourier coefficients. Let $(Z_{k})_{k \in \mathbb{Z}}$ be a sequence of independent standard Gaussian random variables. Then, the series expansion
$$ X_{t} = \sum_{k=1}^\infty \sqrt{\frac{-c_{k}}{2}} \bigg (Z_{k} \sin\frac{k\pi t}{T}  + Z_{-k}\left(1-\cos\frac{k\pi t}{T}\right)    \bigg), \quad t\in[0,T], $$
converges uniformly in $[0,T]$ almost surely and $X_{t}$ is a Gaussian process of type $(A)$. Moreover its rate of convergence is given by
$$
 \mathbf{E} \operatorname*{sup}_{t\in[0,T]}\big|X_{t}-X_{t}^{N} \big| = \operatorname*{\mathcal{O}} \left( N^{-\frac{1-\delta}{2}}\sqrt{\log(N)} \right), 
$$
where $X_{t}^{N}$ is the truncated series of $X_{t}$.
\end{theorem}
%Before proceeding to another generalization we show a very useful lemma
%\begin{lemma}
%Let $(X_{t})_{t\in[0,T]}$ be a stationary process and $\gamma$ the autocovariance function, then:
%\[ \int_{0}^{T}\gamma(t) \mathrm{d}t \geq 0 \]
%\end{lemma}
%\begin{proof}
%On one side
%\[\int_{0}^{T} \int_{0}^{T}\gamma(t-s)\mathrm{d}t \mathrm{d}s = \mathbf{E}\int_{0}^{T} %\int_{0}^{T}X_{t}X_{s}\mathrm{d}t \mathrm{d}s = \mathbf{E}\left(\int_{0}^{T}X_{t}\mathrm{d}t \right)^{2} \]
%On the other side
%\[ \int_{0}^{T} \int_{0}^{T}\gamma(t-s)\mathrm{d}t \mathrm{d}s = \int_{-T}^{T} \int_{v}^{T}\gamma(v)\mathrm{d}t %\mathrm{d}v = \int_{-T}^{T}(T-v)\gamma(v)\mathrm{d}v \]
%Since $\gamma$ is an even function, by combining the two equalities we get that:
%\[ \int_{0}^{T} \gamma(t)\mathrm{d}t = \frac{1}{2T}  \mathbf{E}\left(\int_{0}^{T}X_{t}\mathrm{d}t \right)^{2} \]
%\end{proof}
The next class of interest is a subclass of stationary Gaussian processes. Let $\gamma$ be a function such that $-\gamma$ satisfies $(\star)$ for some $\delta \in [0,1)$, and let $X$ be a centered Gaussian process. We say that $X$ is a Gaussian process of type $(B)$, if it is characterized by the following covariance structure
$$
\forall t,s \in [0,T],\quad \mathbf{E}\left(X_{t}X_{s}\right) =  \gamma(|t-s|).
$$
\begin{theorem}\label{theorem2}
Assume that $\gamma$ is such that $-\gamma$ satisfies $(\star)$ for some $\delta \in [0,1)$, and let $c(\gamma)$ be the sequence of its Fourier coefficients. Let $(Z_{k})_{k \in \mathbb{Z}}$ be a sequence of independent standard Gaussian random variables. If $c_{0} \geq 0$, then the series
\[  X_{t}=\sqrt{c_{0}}Z_{0} +  \sum_{k=1}^\infty \sqrt{c_{k}} \bigg ( \sin\frac{k\pi t}{T} Z_{k} + \cos\frac{k\pi t}{T} Z_{-k}   \bigg), \quad t\in [0,T], \]
converges uniformly in $[0,T]$ almost surely, and $X_{t}$ is a Gaussian process of type $(B)$. Moreover the convergence is rate-optimal, and its rate is given by
$$
 \mathbf{E} \operatorname*{sup}_{t\in[0,T]}\big|X_{t}-X_{t}^{N} \big| = \operatorname*{\mathcal{O}}\left( N^{-\frac{1-\delta}{2}}\sqrt{\log(N)} \right) ,
$$
where $X_{t}^{N}$ is the truncated series of $X_{t}$.
\end{theorem}
One immediate consequence is a series expansion for a stationary fractional Ornstein-Uhlenbeck process $X_{t}$ with $H\in(0,1/2)$, where a stationary fOU is a centered Gaussian process such that
\[\forall s,t \in[0,T], \quad \mathbf{E}\left(X_{s}X_{t}\right) = e^{-|t-s|^{2H}}. \]
The last series expansion was already derived in \citep{PG}.

The framework we are proposing here can also be applied to Gaussian processes that are neither stationary nor with stationary increments. As an example we apply it to another class of Gaussian processes.  
Let $\gamma$ be a function defined on $(0,2T)$, such that $-\gamma$ satisfies $(\star)$ for some $\delta \in [0,1)$, and let $X_{t}$ be a centered Gaussian process. We say that $X_{t}$ is a Gaussian process of type $(C)$, if it is characterized by the following covariance structure
\begin{equation}\label{eq:tss}
\forall t,s \in [0,T], \; \mathbf{E}\left(X_{s}X_{t}\right) = \frac{1}{2}\bigg( \gamma(|t-s|) - \gamma(t+s)\bigg). 
\end{equation}
\begin{theorem}\label{theorem:3}
Assume that $\gamma$ is such that $-\gamma$ satisfies $(\star)$ on $(0,2T)$ for some $\delta \in [0,1)$, and let $c(\gamma)$ be the sequence of its Fourier coefficients on the interval $(0,2T)$. Let $(Z_{k})_{k \geq 1}$ be a sequence of independent standard Gaussian random variables. Then, the series
\[  X_{t}=\sum_{k=1}^\infty \sqrt{c_{k}}  \sin\frac{k\pi t}{2T} Z_{k} , \quad t\in[0,T],  \]
converges uniformly in $[0,T]$ almost surely, and $X_{t}$ is a Gaussian process of type $(C)$. Moreover its rate of convergence is given by
$$
 \mathbf{E} \operatorname*{sup}_{t\in[0,T]}\big|X_{t}-X_{t}^{N} \big| = \operatorname*{\mathcal{O}} \left( N^{-\frac{1-\delta}{2}}\sqrt{\log(N)} \right) ,
$$
where $X_{t}^{N}$ is the truncated series of $X_{t}$.
\end{theorem}

\begin{remark}
One may notice that, in Theorem \ref{theorem:3}, we have considered a $4T$-periodic basis instead of $2T$-periodic because $\forall 0 \leq t,s \leq T, \quad 0\leq s+t \leq 2T$.
\end{remark}
In order to illustrate Theorem \ref{theorem:3}, we give below two examples of corresponding series expansions. 
\begin{example}
Karhunen-Lo\`eve expansion of the Brownian motion.\\
In this example, we consider the function $\gamma(t) = -|t|$ on $[0,T]$. It is easy to check that $\gamma$ satisfies the conditions of Theorem \ref{theorem:3}. One may also notice that this process is a Brownian motion on $[0,T]$ since
\[\forall t,s\in [0,T], \quad \frac{1}{2}\left(-|t-s|+|t+s|\right) = \min(t,s). \]
An explicit evaluation of the sequence $(c_{k})_{k \geq 1}$ gives
\begin{equation}
\begin{aligned}
\forall k \geq 1, \quad c_{k} &= \frac{1}{T}\int_{0}^{2T}-t\cos{\frac{k\pi t}{2T}}\mathrm{d}t \\
 &= \frac{2}{k\pi}\int_{0}^{2T}\sin{\frac{k\pi t}{2T}}\mathrm{d}t \\
 &= \left(1-(-1)^{k}\right) \left(\frac{2}{k\pi}\right)^{2}T. 
\end{aligned}
\end{equation}
Applying Theorem \ref{theorem:3}, it follows that
\[ \forall t\in[0,T],\quad X_{t} = \sqrt{2}\sum_{k=1}^{\infty}\frac{\sqrt{T}}{(k-\frac{1}{2})\pi}Z_{k}\sin{\frac{(k-\frac{1}{2})\pi t}{T}},\]
is a series expansion for Brownian motion on $[0,T]$, where $(Z_{k})_{k \geq 1}$ is a sequence of independent standard Gaussian random variables.

\end{example}
\begin{example}
A new series expansion for the generalized Ornstein-Uhlenbeck process. \\
In this example we consider the non-stationary Ornstein-Uhlenbeck process $(Y_{t})_{t\geq0}$ where $Y_{0}$ is a Gaussian random variable with the following distribution $\mathcal{N}(\mu,\sigma_{0}^{2})$. This process is a Gaussian process characterized by
\[ \forall t\geq 0, \quad \mathbf{E}Y_{t} = \mu e^{-\theta t}+\alpha(1-e^{-\theta t}), \]
and\[\forall s,t \geq 0, \quad \mathbf{E}\left(Y_{s}-\mathbf{E}Y_{s},Y_{t}-\mathbf{E}Y_{t}\right) = \sigma_{0}^{2}e^{-\theta(t+s)}+\frac{\sigma^{2}}{2\theta}\left(e^{-\theta(|t-s|)}-e^{-\theta(t+s)} \right), \]
for some $\theta>0 $ and $\alpha,\sigma \in \mathbb{R}$. By setting $\gamma(t)=\frac{\sigma^{2}}{\theta}e^{-\theta t}$, we can see that $-\gamma$ satisfies conditions of Theorem \ref{theorem:3}. Hence, and applying Theorem \ref{theorem:3}, the following expansion
\begin{equation}\label{eq:lasttt}
\forall t\in [0,T], \quad X_{t} = Y_{0}e^{-\theta t}+\alpha(1-e^{-\theta t}) + \sum_{k=1}^{\infty}\sqrt{c_{k}}Z_{k}\sin{\frac{k\pi t}{2T}},
\end{equation}
is a series expansion for the generalized Ornstein-Uhlenbeck process on $[0,T]$, where $(Z_{k})_{k\geq1}$ is a sequence of independent standard Gaussian random variables, that are also independent from $Y_{0}$. 

An explicit evaluation of the sequence $(c_{k})_{k\geq 1}$ gives
\begin{equation}
\begin{aligned}
\forall k\geq 1, \quad \frac{\theta}{\sigma^{2}}c_{k} &= \frac{1}{T}\int_{0}^{2T}e^{-\theta t}\cos{\frac{k\pi t}{2T}}=Re\left(\frac{1}{T}\int_{0}^{2T}e^{(-\theta+i\frac{k\pi}{2T} )t}\mathrm{d}t \right) \\
&= Re\left( \frac{1 - (-1)^{k}e^{-2\theta T}}{\theta T - \frac{i k \pi}{2}}\right) = \frac{1}{1+\left(\frac{k \pi}{2\theta T}\right)^{2}} \frac{1-(-1)^{k}e^{-2\theta T}}{\theta T}.
\end{aligned}
\end{equation}
The expansion \eqref{eq:lasttt} is easier to use compared to the one known so far that are based on zeros of Bessel functions.
\end{example}

%It is obvious that our series is a gaussian process with zero mean. The previous theorem shows that it has the structure of a fractional Brownian motion on [0,T]. Here are some points %to add to this paper
%\begin{itemize}
%\item This series expansion is not valid for H$>$1/2, because of two reasons. The first is simply because $(c_{k})_{k>0}$ are not positive or of constant sign. we can show that they %are positive for k even and negative for k odd. The second reason is that if we suppose this expansion to be valid, in this case $|c_{k}| \sim \frac{1}{k^{2H+1}}$ and this time %2H+1$>$2 which means that the series is differentiable and its derivative is continuous. This is not possible since our function g in not differentiable in T.   
%\item The expansion series of B converge absolutely and uniformly in [0,T] almost surely. The proof is the same as in the paper of van Zenten by using a Cauchy sequence and the %Kolmogorov's three-series theorem.
%\item {\bf Rate optimality}. In one of van Zenten papers \citep{ZT} they mention and prove that their series expansion is rate-optimal. Since our convergence rates are the same as %those in their paper we can say that our expansion is also rate-optimal.
%\end{itemize}

\section{Application: \textit{Functional quantization}}
Quantization consists in approximating
a random variable taking a continuum of values in $\mathbb{R}$ by a discrete random variable. While vector quantization deals with finite dimensional random variables, functional quantization
extends the concept to the infinite dimensional setting, as it is the case for stochastic processes. 
Quantization of random vectors can be considered as
a discretization of the probability space, providing in some sense the best
approximation to the original distribution. 
The quantization of a random variable $X$ taking values in $\mathbb{R}$ consists in approximating it by the best discrete random variable $Y$ taking finite values in $\mathbb{R}$. If we set $N$ to be the maximum number of values taken by $Y$, the problem is equivalent to minimizing the following error
\begin{equation}\label{eq:str}
\xi_{N}(X) = \left\{ \mathbf{E}\left( X-\text{Proj}_{\Gamma}(X) \right)^{2}, \quad \Gamma \subset \mathbb{R} \textnormal{ such that } |\Gamma|\leq N \right\}  .
\end{equation}
A solution of \eqref{eq:str} is an $L^{2}$-optimal quantizer of $X$. 

For  a multidimensional Gaussian random variable $X$, optimal quantization is expensive. One way to mitigate this cost, is to consider product-quantization, that is to use a cartesian product of one-dimensional optimal-quantizers of each marginal  as in \citep{carlo}. The resulting quantizer is said to be stationary when marginals of $X$ are independent. In \citep{LAST}, it is shown that Karhunen-Lo\`eve product-quantization, while it is sub-optimal, remains rate-optimal in the case of Gaussian processes.   

We consider now a continuous Gaussian process $(X_{t})_{t\in[0,T]}$ such that $\int_{0}^{T}\mathbf{E}|X_{t}^{2}|\mathrm{d}t <\infty$, and its expansion
\[\forall t\in[0,T], \quad  X_{t}=\sum_{i=0}^{\infty}\lambda_{i}e_{i}(t)Z_{i},\]
where $(\lambda_{i})_{i\in \mathbb{N}}$ is a sequence of real numbers such that $\sum_{i=0}^{\infty}\lambda_{i}^{2}<\infty$, $(e_{i})_{i\in\mathbb{N}}$ is an orthonormal sequence of continuous functions, and $(Z_{i})_{i\in\mathbb{N}}$ a sequence of independent standard Gaussian random variables. Notice that the Karhunen-Lo\`eve expansion is a special case of what we are introducing. In this case the error induced by replacing the process by a rate-optimal quantizer of its truncation up to order $m$ is given by
\[ \xi_{N}(X)^{2} = \int_{0}^{T}\mathbf{E}\left(X_{t}-\sum_{i=0}^{m}\lambda_{i}e_{i}(t)Y_{i}\right)^{2} \mathrm{d}t,  \]
where $\forall \; 0 \leq i \leq m, \; Y_{i}$ is an optimal quantizer of $Z_{i}$ taking $N_{i}$ values and $\prod_{i=0}^{m}N_{i}\leq N$. More precisely we get that
\[ \xi_{N}(X)^{2} = \sum_{i=m+1}^{\infty}\lambda_{i}^{2} + \sum_{i=0}^{m}\xi_{N_{i}}(\mathcal{N}(0,\lambda_{i}^{2})).\] 
If moreover $\lambda_{N}^{2}\asymp\frac{1}{N^{\delta}}$, with $1<\delta < 3$, it is shown in \citep{LH} that, the optimal product-quantization of level $N$ is achieved when the dimension of the quantizer $m$ is of order $\log{N}$ and that it satisfies
\[ \xi_{N}(X)\asymp(\log{N})^{\frac{1-\delta}{2}}.\]
When the basis $(e_{k})_{k \in \mathbb{N}}$, chosen in the series expansion, is not orthonormal, an alternative rate-optimal quantization method is presented in \citep{JG}. The idea consists in truncating the series up to the optimal order $m\asymp \log{N}$ and then considering the finite-dimensional covariance operator $K^{m}$ of the truncation given by
$$
\forall t,s \in [0,T],\quad K^{m}(t,s) = \sum_{i=0}^{m} \lambda_{i}^{2}e_{i}(t)e_{i}(s).
$$
More specifically, consider $H$ a  linear subspace of $L^2$ defined by $H = \text{Span}((e_{i})_{0\leq i \leq m}) $. The operator  $T_{K^{m}}$ is given by
\[ T_{K^{m}} : \begin{cases}& H \to H \\ & f \to \int_{0}^{T}K^{m}(s,.)f(s)\mathrm{d}s. \end{cases}\]
$T_{K^{m}}$ is clearly an endomorphism. 
Unlike in the proof the Karhunen-Lo\`eve theorem, in this case we deal with a linear and symmetric operator in finite dimension. Hence there exists $(\mu_{i}^{m})_{0\leq i \leq m }$ a sequence of positive real numbers and  $(f_{i}^{m})_{0\leq i \leq m }$ an orthonormal basis of $H$ such that
\[\forall t,s \in[0,T] , \quad K^{m}(t,s)=\sum_{i=0}^{m}\mu_{i}^{m}f_{i}^{m}(t)f_{i}^{m}(s). \]
 We can then assert that there exists $(Y_{i}^{m})_{0\leq i \leq m}$ a sequence of independent standard Gaussian random variables such that
$$
\forall t \in [0,T], \quad \sum_{i=0}^{m}\lambda_{i}e_{i}(t)Z_{i} = \sum_{i=0}^{m}\sqrt{\mu_{i}^{m}}f_{i}^{m}(t)Y_{i}^{m} \quad a.s.
$$
Following the same argument as in \citep{JG}, if we set $m \asymp \log{N}$ and replace the process by a rate-optimal quantizer of $\sum_{i=0}^{m}\sqrt{\mu_{i}^{m}}f_{i}^{m}(t)Y_{i}^{m} $, we will get the quadratic quantization error 
$$
\xi_{N}(X)^{2} = \mathcal{O}\left(\sum_{i=m+1}^{\infty}\lambda_{i}^{2} + \sum_{i=0}^{m}\xi_{N_{i}}(\mathcal{N}(0,\mu_{i}^{2}))\right).
$$
Moreover, this error is rate-optimal. Observe that for fBm $\lambda_i^2 = b_i$ and  
\[
 \sum_{i=m+1}^{\infty}b_{i} \sim \sum_{i=m+1}^{\infty}\nu_{i},
\]
according to Lemma \ref{lem:compare}. Hence the error due to the remainder behaves asymptotically as if the Karhunen-Lo\`eve expansion were used, which is another advantage of our expansion.
As an illustration, we give a rate-optimal quantization for both fBm and generalized Ornstein-Uhlenbeck with $T=1$ and $N=20$.\\

\begin{figure}[ht]
\centering
\begin{subfigure}{.5\textwidth}
  \centering
  \includegraphics[width=.8\linewidth]{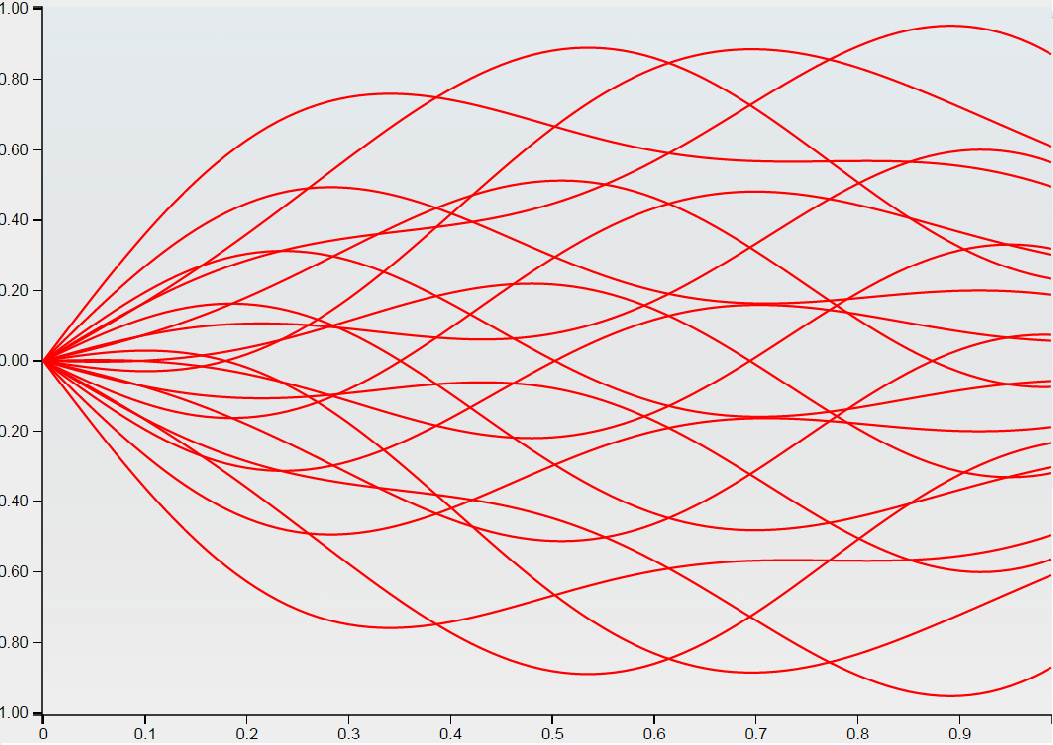}
  \caption{generalized OU $\theta=2,\alpha=0$}
  \label{fig:sub1}
\end{subfigure}%
\begin{subfigure}{.5\textwidth}
  \centering
  \includegraphics[width=.8\linewidth]{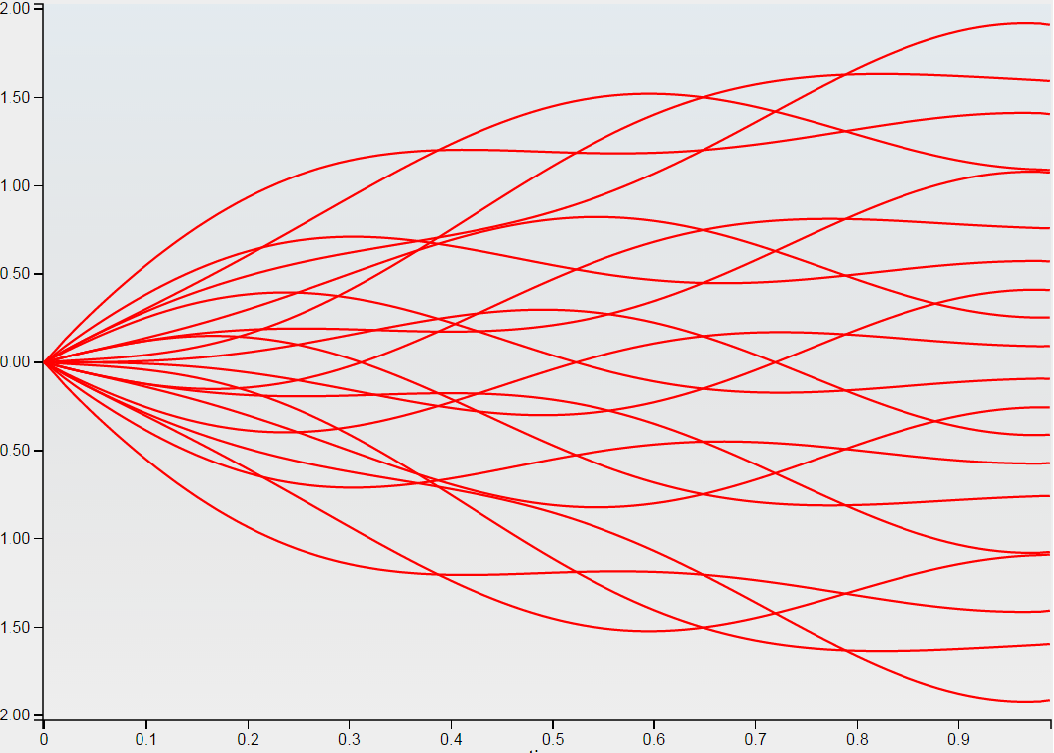}
  \caption{fBm $H=0.4$}
  \label{fig:sub2}
\end{subfigure}
\caption{Product quantization of a centered Ornstein-Uhlenbeck process, starting from $Y_0 = 0$ (left), and a fBm (right).}
\label{fig:test}
\end{figure}
\section{Conclusion}
This paper presents a new framework to derive series expansions for a specific class of Gaussian processes based on harmonic analysis. One of the main results is a new, simple and rate-optimal series expansion for fractional Brownian motion. One of the advantages of this expansion is that its coefficients are easily computed which can reduce the complexity of simulation, especially for the case $H<1/2$ where no other simple trigonometric series expansion is known. Our approach is general and gives series expansions for a large class of Gaussian processes, in particular to the generalized Ornstein-Uhlenbeck process. The application to quantization is interesting in particular for fBm, where the series basis is not orthonormal. In this case, we show how to deal with non-orthonormality and construct a rate-optimal quantizer.\\

%\input{oki.bbl}

%\begin{thebibliography}{99}
%\bibitem{ZT}
%  {\sc K.\, Dzhaparidze and H. \,van Zanten},
%  {\em A series expansion of fractional Brownian motion},
%    {\rm  Probab.\ Theory Relat.\ Fields} {\bf  130 } (2004), 39–55. 
%\bibitem{GL}
%  {\sc E.\, Igloi},
%  {\em A Rate-Optimal Trigonometric Series Expansion of the Fractional Brownian Motion},
%    {\rm  Electron. J. Probab.} {\bf  10 } (2005) , EJP287.
%\bibitem{KH}
%  {\sc T.\, Kuhn and W. \. Linde.},
%  {\em Optimal series representation of fractional Brownian sheets},
%    {\rm  Bernoulli} {\bf  8 } (2002) , 669–696. Math. Review MR1935652 (2003m:60131)
%\bibitem{PG}
%  {\sc H.\, Luschgy and G. \. Pages.},
%  {\em Expansions for Gaussian processes and Parseval frames},
%    {\rm } {\bf   } (2009)
    
%\bibitem{LH}
%  {\sc H.\, Luschgy and G. \. Pages.},
%  {\em Functional quantization of Gaussian processes},
%    {\rm } {\rm Journal of functional Analysis   } (2002)
%\bibitem{JG}
%  {\sc S.\, Junglen and H. \. Luschgy.},
%  {\em A Constructive Sharp Approach to Functional Quantization of Stochastic Processes},
%    {\rm } {\rm Journal of functional Analysis   } (2010)
%\bibitem{RB}
%  {\sc J.\, Gatheral , T.\,Jaisson and M. \,Rosenbaum},
%  {\em Volatility is rough},
%    {\rm  arXiv} {1410.3394v1} (2014) 

%\end{thebibliography}

\begin{appendix}\label{appendix:1}
\section{Proofs of main Theorems}
\label{app:thm}
\begin{proof}[Proof of Proposition \ref{pr:1}]
We first prove that $\gamma$ is integrable. The function $\gamma'$ is continuous on $(0,T]$, positive and for $\delta \in (0,2)$ we have that
there exists $M>0$ and $\epsilon>0$ such that
\begin{equation}\label{eq:pss2}
\forall x \in(0,\epsilon), \quad 0 \leq \gamma'(x)\leq \frac{M}{x^{\delta}}.
\end{equation}
By integrating \eqref{eq:pss2}, we get
$$\gamma(x)=\mathcal{O}\left(1+x^{1-\delta}\right),$$
as $x \to 0^{+}$.
The last result holds also for $\delta=0$. Since $0\leq \delta<2$ and $\gamma$ is continuous, it comes out that $\gamma$ is integrable on (0,T]. It follows that $c(\gamma)$ is well defined.

Before showing the second part, one may first notice that $\gamma'$ is positive and decreasing 
since $\gamma$ is concave and increasing. By a change of variable in \eqref{eq:defck}, we get
$$
 \forall k  \geq 1, \quad c_{k}=\frac{2}{T}\frac{T}{k\pi} \int_0^{k\pi} \gamma\left(\frac{Tu}{k\pi}\right)\cos(u)\mathrm{d}u. 
$$
Since 
$$\gamma(u)\sin(u) = \mathcal{O}\left(\sin(u)+u^{1-\delta}\sin(u) \right), $$
as $u \to 0^{+}$, then 
$$ \lim_{u\to 0^{+}} \gamma(u)\sin(u) = 0. $$
Using integration by parts we obtain, for all $k \geq 1$,
\begin{equation}\label{eq:prt}
\begin{aligned}
        c_{k} &= - \frac{2}{T}\bigg(\frac{T}{k\pi}\bigg)^{2}\int_0^{k\pi} \gamma'\left(\frac{Tu}{k\pi}\right)\sin(u)\mathrm{d}u \\
        &= - \frac{2}{T}\bigg(\frac{T}{k\pi}\bigg)^{2}\;\sum_{n=0}^{k-1}(-1)^{n}\int_0^{\pi} \gamma'\bigg(\frac{T(u+n\pi)}{k\pi}\bigg)\sin(u)\mathrm{d}u. 
       \end{aligned}
\end{equation}
For $0 \leq n < k$, we define 
$$v_{k,n}:=\int_0^{\pi} \gamma'\bigg(\frac{T(u+n\pi)}{k\pi}\bigg)\sin(u)\mathrm{d}u.$$
It is immediate that,
$\forall k  \geq 1, \;(v_{k,n})_{n<k}$ is nonnegative and decreasing with respect to $n$.
Regrouping each pair of elements in \eqref{eq:prt}, we get
$$c_{k} = - \frac{2}{T}\bigg(\frac{T}{k\pi}\bigg)^{2}\left(\sum_{n=0}^{\lfloor\frac{k}{2}\rfloor-1}(v_{k,2n}-v_{k,2n+1})
 + \frac{1-(-1)^{k}}{2}v_{k,k-1}\right).
 $$
It follows that $c_{k} \leq 0,\forall k  \geq 1$.
For the last point, we use again the second part of \eqref{eq:prt} and get
$$
c_{k}= - \frac{2}{T}\bigg(\frac{T}{k\pi}\bigg)^{2}\sum_{n=0}^{k-1}(-1)^{n}v_{k,n}. 
$$
Since the sequence $\left( (-1)^{n}v_{k,n}\right)_{n<k}$ has alternating signs and a decreasing modulus, it turns out that
$$ |c_{k}| \leq \frac{2T^{2}}{k^{2}\pi^{2}} v_{k,0} \leq \frac{2T^{2}}{k^{2}\pi^{2}} \int_0^{\pi} \gamma'\bigg(\frac{Tu}{k\pi}\bigg)\sin(u)\mathrm{d}u. $$
In order to conclude, it is enough to prove that
$$ \int_0^{\pi} \gamma'\bigg(\frac{Tu}{k\pi}\bigg)\sin(u)\mathrm{d}u = \mathcal{O}(k^{\delta}). $$
Since $\gamma\in\Gamma$, we can check that $x \to x^{\delta}\gamma'(x)$ is uniformly bounded on $[0,T]$. Let $M$ be this uniform bound, then
$$ 0 \leq \int_0^{\pi} \gamma'\bigg(\frac{Tu}{k\pi}\bigg)\sin(u)\mathrm{d}u \leq M\left(\frac{k\pi}{T}\right)^{\delta}\int_{0}^{\pi}\frac{\sin(u)}{u^{\delta}} \mathrm{d}u. $$
As $\delta$ belongs to $[0,2)$, it follows that
$$ |c_{k}| = \mathcal{O}(k^{\delta-2}). $$
This concludes the proof.

\end{proof}
\begin{proof}[Proof of Theorem \ref{thm:1}]
We will use the following standard bound on the maximum of centered Gaussian random variables. If $X_{1},\dots,X_{M}$ are centered Gaussian random variables, then there exists $c >0$, such that
\begin{equation}\label{eq:brns}
\mathbf{E}\underset{1\leq i \leq M}{\max}|X_{i}| \leq c\sqrt{\log{M}}\underset{1\leq i \leq M}{\max}\sqrt{\mathbf{E}X_{i}^{2}}.
\end{equation}
We denote by $v_{k}(t):=\lambda_{k}e_{k}(\frac{k\pi t}{T})Z_{k}$ for $k\in\mathbb{N}$ and $t\in[0,T]$. 
The proof is divided in two parts. We first show that, for some $A>0$, we have
\begin{equation}\label{eq:best}
\forall n \geq 1, \quad \mathbf{E}\underset{t\in[0,T]}{\sup}\left|\sum_{k=2^n}^{2^{n+1}-1}v_{k}(t) \right| \leq A\sqrt{n}2^{-nH}. 
\end{equation}
Let $N\geq 1$, for all $ 0\leq j \leq N-1$, we denote by $I_{j}=\left[j\frac{T}{N},(j+1)\frac{T}{N}\right]$ and $t_{j}$ the corresponding center i.e. $t_{j}=(j+1/2)\frac{T}{N}$. Let $n\geq 1$, we have 
\begin{equation}\label{eq:tz}
    \begin{aligned}
    \mathbf{E}\underset{t\in[0,T]}{\sup}\left|\sum_{k=2^n}^{2^{n+1}-1}v_{k}(t) \right| &= \mathbf{E}\underset{0\leq j<N}{\sup}\;\underset{t\in I_{j}}{\sup}\left|\sum_{k=2^n}^{2^{n+1}-1}v_{k}(t) \right|\\
    &\leq \mathbf{E}\underset{0\leq j<N}{\sup}\left|\sum_{k=2^n}^{2^{n+1}-1}v_{k}(t_{j}) \right| +  \mathbf{E}\underset{0\leq j<N}{\sup}\;\underset{t\in I_{j}}{\sup}\left|\sum_{k=2^n}^{2^{n+1}-1}\left(v_{k}(t)-v_{k}(t_{j})\right) \right|.
    \end{aligned}
\end{equation}
Using \eqref{eq:brns} we get 
\begin{equation}\label{eq:deux}
\begin{aligned}
    \mathbf{E}\underset{0\leq j<N}{\sup}\left|\sum_{k=2^n}^{2^{n+1}-1}v_{k}(t_{j}) \right| &\leq c\sqrt{\log{N}}\underset{0\leq             j<N}{\sup}\sqrt{\mathbf{E}\left|\sum_{k=2^n}^{2^{n+1}-1}v_{k}(t_{j}) \right|^{2}}\\
    &\leq  c\sqrt{\log{N}}\underset{0\leq j<N}{\sup}\sqrt{\sum_{k=2^n}^{2^{n+1}-1}\mathbf{E}v_{k}(t_{j})^{2} } \\
    &\leq C'\sqrt{\log{N}}2^{-nH},
\end{aligned}
\end{equation}
for some $C'>0$. The last inequality comes from the fact that $\mathbf{E}v_{k}(t_{j})^{2}\leq \lambda_{k}^{2}||e_{k}||_{\infty}^{2} \leq \frac{C}{k^{1+2H}} $, for some $C>0$.\\
For the second part of \eqref{eq:tz}, we have
\begin{equation}\label{eq:last}
     \mathbf{E}\underset{0\leq j<N}{\sup}\;\underset{t\in I_{j}}{\sup}\left|\sum_{k=2^n}^{2^{n+1}-1}\left(v_{k}(t)-v_{k}(t_{j})\right) \right| \leq  \mathbf{E}\underset{0\leq j<N}{\sup}\;\sum_{k=2^n}^{2^{n+1}-1}\underset{t\in I_{j}}{\sup}\left|v_{k}(t)-v_{k}(t_{j})\right|.
\end{equation}
Observing that $\forall t\in I_{j}$ we have $|t-t_{j}|\leq\frac{T}{N}$, we get that
\begin{equation}\label{eq:ret}
    \begin{aligned}
    \underset{t\in I_{j}}{\sup}
    \left|v_{k}(t)-v_{k}(t_{j})\right| 
    &\leq |\lambda_{k}||Z_{k}|
    \left|e_{k}\left( \frac{k\pi t}{T}\right) -e_{k} \left( \frac{k\pi t_{j}}{T} \right)\right|\\
    &\leq C' k^{\frac{1}{2}-H} |Z_{k}|\frac{\pi}{N}.
    \end{aligned}
\end{equation}
Replacing in \eqref{eq:last}, it follows that
\begin{equation}\label{eq:un}
\begin{aligned}
\mathbf{E}\underset{0\leq j<N}{\sup}\;\underset{t\in I_{j}}{\sup}\left|\sum_{k=2^n}^{2^{n+1}-1}\left(v_{k}(t)-v_{k}(t_{j})\right) \right| &\leq \frac{C''}{N} \sum_{k=2^n}^{2^{n+1}-1}k^{\frac{1}{2}-H}\\
&\leq \frac{C^{*}}{N}2^{n(\frac{3}{2}-H)}.
\end{aligned}
\end{equation}
Combining \eqref{eq:deux} and \eqref{eq:un} we deduce that
\begin{equation}\label{eq:uuu}
 \mathbf{E}\underset{t\in[0,T]}{\sup}\left|\sum_{k=2^n}^{2^{n+1}-1}v_{k}(t) \right| \leq C\sqrt{\log{N}}2^{-nH} + \frac{C^{*}}{N}2^{n(\frac{3}{2}-H)}.
\end{equation}
Replacing $N=2^{2n}$ in \eqref{eq:uuu}, we prove \eqref{eq:best}. 

The previous result holds even if we replace $\left|\sum_{k=2^n}^{2^{n+1}-1}v_{k}(t) \right|$ by $\left|\sum_{k=M}^{2^{n+1}-1}v_{k}(t) \right|$ for some $M\in[2^{n},2^{n+1}-1]$. Let $N$ be a positive integer. By taking $m = \lfloor \log{N}/\log{2}\rfloor$, we get
\begin{equation}
\mathbf{E} \operatorname*{sup}_{t\in[0,T]}\left|\sum_{k=N+1}^{\infty} v_{k}(t)\right|\leq \mathbf{E} \operatorname*{sup}_{t\in[0,T]}\left|\sum_{k=N+1}^{2^{m+1}-1}v_{k}(t)\right| + \sum_{i=m+1}^{\infty}\mathbf{E}\underset{t\in[0,T]}{\sup} \left|\sum_{k=2^i}^{2^{i+1}-1}v_{k}(t)  \right|.
\end{equation}
We can conclude, using \eqref{eq:best}, that
\begin{equation}
\begin{aligned}
\mathbf{E} \operatorname*{sup}_{t\in[0,T]}\left|\sum_{k=N+1}^{\infty} v_{k}(t)\right| \leq A\sum_{i=m}^{\infty}\sqrt{i}2^{-iH} \leq A'\sqrt{m}2^{-mH}.
\end{aligned}
\end{equation}
It suffices to observe that $2^m \leq N < 2^{m+1}$, to obtain the rate of the uniform convergence. The uniform tightness implies that $\sum_{k=0}^{N}v_{k}$ has a weak limit in C[0,T] the space of continuous functions on [0,T]. We remind the reader that we endow this space with the supremum metric. By the It\^{o}-Nisio theorem, we get, as in \citep{ITO}, that the process $\sum_{k=0}^{N}v_{k}$ converges in $C[0,T]$ almost surely.
\end{proof}
\begin{proof}[Proof of Theorem \ref{proof:thm:1}]
For $H \in \left(0,1/2\right)$, we denote by $\gamma(t):=|t|^{2H}$. It is easy to check that $\gamma$ satisfies $(\star)$. Using Proposition \ref{pr:1}, the above series is well-defined since $\forall k\geq1, \;c_{k} \leq 0$ and hence we can write $\sqrt{-c_k}$ (remember that we are assuming the series to be convergent).
Because of the independence between the Gaussian random variables $Z$, it follows immediately that
\begin{equation}
\begin{aligned}
\mathbf{E}(B_{s}B_{t}) &= \sum_{k=1}^\infty -\frac{c_{k}}{2} \left ( \sin\frac{k\pi s}{T}\sin\frac{k\pi t}{T} + \left(1-\cos\frac{k\pi s}{T}\right)\left(1-\cos\frac{k\pi t}{T}\right) \right) \\
 &= \sum_{k=1}^\infty -\frac{c_{k}}{2} \left( 1-\cos\frac{k\pi s}{T}-\cos\frac{k\pi t}{T}+\cos\frac{k\pi (t-s)}{T}   \right) \\
 &= \sum_{k=1}^\infty \frac{c_{k}}{2} \left ( \left(\cos\frac{k\pi s}{T}-1\right)+\left(\cos\frac{k\pi t}{T}-1\right)-\left(\cos\frac{k\pi (t-s)}{T}-1\right)   \right).
\end{aligned}
\end{equation}
We conclude using Lemma \ref{lem:1}.
\end{proof}
\begin{proof}[Proof of Theorem \ref{pr:thm:2}]
By considering $\gamma(t) = -2H(2H-1)t^{2H-2}$, we notice that $\gamma$ satisfies $(\star)$ for $H \in (1/2,1)$. Moreover 
$$|\gamma'(t)| = \mathcal{O}\left(\frac{1}{t^{3-2H}}\right),$$
as $t \to 0^{+}$.
Since $1<3-2H<2$, we get using Proposition \ref{pr:1} that $ c_{k} \leq 0$ for all $k\geq 1$, and $|c_{k}| = \mathcal{O} \left(\frac{1}{k^{2H-1}}\right)$. We also obtain, using Lemma \ref{lem:2} on the function $t\to|t|^{2H}$, that
$$ \frac{2}{T}\int_{0}^{T} \left(t^{2H} - HT^{2H-2} t^{2}\right) \cos\left(\frac{k\pi t}{T}\right)\mathrm{d}t = \left(\frac{T}{k\pi}\right)^{2} c_{k}. $$
Since $\frac{|c_{k}|}{k^{2}} = \mathcal{O}\left( \frac{1}{k^{2H+1}} \right)$, the Fourier series converges uniformly and we can apply Lemma \ref{lem:1} to get
$$\forall t\in [-T,T], \quad |t|^{2H} - H T^{2H-2} t^{2} = \sum_{k=1}^{\infty}\left(\frac{T}{k\pi} \right)^{2}c_{k} \left(\cos{\frac{k\pi t}{T}}-1 \right).$$
Since $c_{k}\leq 0$ the series expansion  is well defined and we have
\begin{equation}
\begin{aligned}
\mathbf{E}(B_{t}B_{s}) &= HT^{2H-2} t s - \frac{1}{2}\sum_{k=1}^{\infty}\left(\frac{T}{k\pi}\right)^{2}c_{k}\left(\sin{\frac{k\pi t}{T}}\sin{\frac{k\pi s}{T}}+\left(1-\cos{\frac{k\pi t}{T}} \right)\left(1-\cos{\frac{k\pi s}{T}} \right) \right) \\
&= HT^{2H-2} t s - \frac{1}{2}\sum_{k=1}^{\infty}\left(\frac{T}{k\pi}\right)^{2}c_{k}\left(1-\cos{\frac{k\pi t}{T}}-\cos{\frac{k\pi s}{T}} +\cos{\frac{k\pi (t-s)}{T}} \right) \nonumber \\
&= \frac{1}{2}\left( HT^{2H-2}(t^{2}+s^{2}-(t-s)^{2}) + \sum_{k=1}^{\infty}\left(\frac{T}{k\pi}\right)^{2}c_{k}\left(\cos{\frac{k\pi t}{T}}+\cos{\frac{k\pi s}{T}} -\cos{\frac{k\pi (t-s)}{T}}-1 \right)\right)\nonumber \\
&=\frac{|t|^{2H} + |s|^{2H} - |t-s|^{2H}}{2}. 
\end{aligned}
\end{equation}
\end{proof}

\begin{proof}[Proof of Proposition \ref{prop:11}]
As in previous proofs, we use the independence of Gaussian random variables $(Z_{k})_{k\in\mathbb{Z}}$. It is straightforward to see that, for all  $t \in[0,T]$,
\[  \mathbf{E}\big(B_{t}-B_{t}^{N} \big)^2 = \sum_{k>N} -\frac{c_{k}}{2} \bigg ( (\sin\frac{k\pi t}{T})^2 + (1-\cos\frac{k\pi t}{T})^2   \bigg) 
   =  \sum_{k>N} -c_{k} \bigg ( 1-\cos\frac{k\pi t}{T}   \bigg).  \] 
Hence \[  \operatorname*{sup}_{t\in[0,T]} \sqrt{\mathbf{E}\big(B_{t}-B_{t}^{N} \big)^2} \leq \sqrt{\sum_{k>N} |c_{k}|}.  \] 
Since $|c_{k}| = \mathcal{O}\left( \frac{1}{k^{2H+1}}\right)$, the result follows.
\end{proof}

\begin{proof}[Proof of Proposition \ref{thm:opt}]
We only need to prove that the rate of convergence of the above series is faster than $N^{-H}\sqrt{\log(N)}$ since the latter is the optimal rate of convergence for fBm as shown in \citep{KH}.\
By truncating  the series, we have 
\[\forall t \in [0,T],\quad B_{t}- B_{t}^{N} =  \sum_{k=N+1}^{\infty} \sqrt{-\frac{c_{k}}{2}} \left( Z_{k}  \sin\frac{k\pi t}{T} + Z_{-k} \left(1-\cos\frac{k\pi t}{T}\right)    \right).  \]
Since $\sqrt{-\frac{c_{k}}{2}}=\mathcal{O}\left(\frac{1}{k^{H+1/2}}\right)$, and using the fact that  $t\to\sin(t)$ and $t\to1-\cos(t)$ are 1-Lipschitz functions we can directly use Theorem \ref{thm:1} to conclude the proof.
\end{proof}

\begin{proof}[Proof of Theorem \ref{theorem1}]
Applying Proposition \ref{pr:1}, we obtain that $ c_{k} \leq 0$ for $k\geq 1$. Hence, the series is well defined. Moreover, we also have that $|c_{k}| =  \mathcal{O}\left(\frac{1}{k^{2-\delta}}\right)$.
Since $\delta\in[0,1)$ we can use Theorem \ref{thm:1} to get that
\[ \mathbf{E} \underset{t \in [0,T]}{\sup}\left| \sum_{k=N+1}^\infty \sqrt{\frac{-c_{k}}{2}} \bigg (Z_{k}  \sin\frac{k\pi t}{T} + Z_{-k} \left(1-\cos\frac{k\pi t}{T}\right)    \bigg) \right| =\mathcal{O}\left(\frac{\sqrt{\log(N)}}{N^{\frac{1-\delta}{2}}}\right),\]
and that the series converges almost surely and uniformly in $[0,T]$. It follows that $X_{t}$ is a centered Gaussian process. Its covariance function is given by
\[\forall s,t \in[0,T], \quad \mathbf{E}(X_{s}X_{t}) = \sum_{k=1}^{\infty}\frac{-c_{k}}{2}\left(1-\cos{\frac{k\pi t}{T}} - \cos{\frac{k\pi s}{T}} +\cos{\frac{k\pi(t-s)}{T}} \right). \]
We can conclude using Lemma \ref{lem:1} that
\[\forall s,t \in[0,T], \quad \mathbf{E}(X_{s}X_{t}) = \frac{1}{2}\left(\gamma(t)+\gamma(s)-\gamma(|t-s|)\right). \]
Hence $X_{t}$ is  a Gaussian process of type $(A)$.
\end{proof}

\begin{proof}[Proof of Theorem \ref{theorem2}]
Using the same steps as for Theorem \ref{theorem1}, we get that the series is well defined, and that it converges uniformly almost surely. Hence $X_{t}$ is a Gaussian process. Moreover we have
\[\forall s,t \in[0,T], \quad \mathbf{E}(X_{s}X_{t}) =\sum_{k=0}^{\infty}c_{k}\cos{\frac{k\pi (t-s)}{T}} =\gamma(|t-s|). \]
It follows that $X_{t}$ is a Gaussian process of type $(B)$. Since the basis $(e_{k})_{k\in\mathbb{Z}}$ used in this expansion is orthogonal, we may apply the same argument as in Proposition 4 in \citep{PG}, and deduce that the convergence is rate-optimal.
\end{proof}

\begin{proof}[Proof of Theorem \ref{theorem:3}]
Using the same steps as for Theorem \ref{theorem1}, we get that the series is well defined, and that it converges uniformly almost surely. Hence $X_{t}$ is a Gaussian process. Moreover we have
\[ \forall t,s \in [0,T], \; \mathbf{E}(X_{s}X_{t}) = \sum_{k=1}^\infty c_{k}  \bigg(\sin\frac{k\pi t}{2T}\sin\frac{k\pi s}{2T}\bigg)= \frac{1}{2}\sum_{k=1}^\infty c_{k}  \bigg(\cos{\frac{k\pi(t-s)}{2T}}-\cos{\frac{k\pi(t+s)}{2T}}\bigg). \]
As a consequence we get that
\[ \forall t,s \in [0,T], \; \mathbf{E}(X_{s}X_{t}) = \frac{1}{2}\bigg( \gamma(|t-s|) - \gamma(t+s)\bigg). \]
It follows that $X_{t}$ is a Gaussian process of type $(C)$.
\end{proof}

\section{Proofs of auxiliary Lemmas}
\label{app:lemma}
\begin{proof}[Proof of Lemma \ref{lem:1}]
Let  $g:\mathbb{R}\to\mathbb{R}$ be a function such that
$$\forall t \in [-T,T], \quad g(t)=\gamma(|t|).$$
Extending $g$ into a $2T$-periodic function, it can be defined on $\mathbb{R}$. Since $g$ is an even function, its Fourier expansion is given by 
\begin{equation}\label{eq:re}
\forall t \in [-T,T], \quad g(t) = \sum_{k=0}^\infty c_{k}\cos\frac{k\pi t}{T},  
\end{equation}
where $c_{0} = \frac{1}{T}\int_0^T \gamma(t)\mathrm{d}t  \;$ and   $c_{k} = \frac{2}{T}\int_0^T \gamma(t)\cos\frac{k\pi t}{T}\mathrm{d}t$, for $k \geq 1$.
Using Proposition \ref{pr:1}, we have 
$$|c_{k}| =\mathcal{O}\left( \frac{1}{k^{2-\delta}}\right).$$
Since $0\leq\delta<1$, the Fourier expansion of $g$ is normally convergent and hence converges uniformly. Replacing $t$ by $0$ in \eqref{eq:re} we get that
$$ c_{0} = \gamma(0) - \sum_{k=1}^\infty c_{k}.$$ 
It follows that
$$  g(t) = \gamma(0) + \sum_{k=1}^\infty c_{k}\bigg(\cos\frac{k\pi t}{T}-1\bigg). $$
\end{proof}
\begin{proof}[Proof of Lemma \ref{lem:2}]
The function $f$ is constructed in such a way that $f'(0)=f'(T)=0$. For all $k \geq 1$, integration by parts yields
\begin{equation}
\begin{aligned}
\int_{0}^{T} f(t) \cos\left(\frac{k\pi t}{T}\right)\mathrm{d}t &= \frac{T}{k\pi}\left[f(t)\sin{\left( \frac{k\pi t}{T}\right)} \right]_{0}^{T}  - \frac{T}{k\pi} \int_{0}^{T} f'(t) \sin\left(\frac{k\pi t}{T}\right)\mathrm{d}t \nonumber\\
&= \left(\frac{T}{k\pi}\right)^{2}\left[f'(t)\cos{\left( \frac{k\pi t}{T}\right)} \right]_{0}^{T} - \left(\frac{T}{k\pi}\right)^{2}\int_{0}^{T} f''(t) \cos\left(\frac{k\pi t}{T}\right)\mathrm{d}t \nonumber \\
&= - \left(\frac{T}{k\pi}\right)^{2}\int_{0}^{T} \gamma''(t) \cos\left(\frac{k\pi t}{T}\right)\mathrm{d}t. 
\end{aligned}
\end{equation}
The last equality is due to the fact that cosine harmonics are orthogonal to constants on $[0,T]$. 
\end{proof}
\begin{proof}[Proof of Lemma \ref{lem:compare}]
Let $(a_k)_k$ and $(b_k)_k$ the sequences defined in \eqref{eq:inglot1}-\eqref{eq:inglot2} and \eqref{eq:simo1}-\eqref{eq:simo2} respectively. It comes out that
\begin{align*}
    \frac{b_0}{a_0} &= \frac{2H(2H-1)\Gamma(H-\frac{1}{2})\Gamma(\frac{3}{2}-H)}{2\Gamma(2-2H)}\\
    &= \frac{\Gamma(2H+1)\Gamma(H-\frac{1}{2})\Gamma(\frac{3}{2}-H)}{2\Gamma(2-2H)\Gamma(2H-1)}\\
    &= \frac{\Gamma(2H+1)\sin(2\pi(H-1/2)) }{2 \sin(\pi(H-1/2))}\\
    &= \Gamma(2H+1) \sin(\pi H).
\end{align*}
For the asymptotic behaviour, we can write $b_k$ as
\begin{align*}
b_{k} &= \frac{2H(2H-1)}{(k\pi)^{2}}\int_0^1 t^{2H-2}\cos(k\pi t)\mathrm{d}t\\
&= 2H(2H-1)(k\pi)^{-2H-1} \int_0^{k \pi} t^{2H-2}\cos( t)\mathrm{d}t \\
&= 2H(2H-1)(k\pi)^{-2H-1} \Re(i\exp^{-i\pi H}\gamma(2H-1,ik\pi)).
\end{align*}
It follows that
\begin{align*}
\underset{k \to \infty}{\lim} b_{k}(k\pi)^{2H+1} &=  2H(2H-1) \sin(\pi H)\Gamma(2H-1)\\
&=  \Gamma(2H+1)\sin(\pi H).
\end{align*}
That concludes the proof.
\end{proof}

\end{appendix}

\section*{Acknowledgments}
{This project started while the author was an intern at Bloomberg. The author would like to thank Bruno Dupire and Sylvain Corlay for their valuable insights and discussions. The author is also grateful to A.B. Tsybakov for his comments on early drafts of this work}

\bibliographystyle{imsart-number} 
%\bibliography{ok}

\begin{thebibliography}{21}
% BibTex style file: imsart-number.bst, 2017-11-03
% Default style options (sort=1,type=number).
% Used options (sort=1,type=number).

\bibitem{ayache}
\begin{barticle}[author]
\bauthor{\bsnm{Ayache},~\bfnm{Antoine}\binits{A.}} \AND
  \bauthor{\bsnm{Taqqu},~\bfnm{Murad~S}\binits{M.~S.}}
(\byear{2003}).
\btitle{Rate optimality of wavelet series approximations of fractional Brownian
  motion}.
\bjournal{Journal of Fourier Analysis and Applications}
\bvolume{9}
\bpages{451--471}.
\end{barticle}
\endbibitem

\bibitem{bronski1}
\begin{barticle}[author]
\bauthor{\bsnm{Bronski},~\bfnm{Jared~C}\binits{J.~C.}}
(\byear{2003}).
\btitle{Asymptotics of Karhunen-Loeve eigenvalues and tight constants for
  probability distributions of passive scalar transport}.
\bjournal{Communications in mathematical physics}
\bvolume{238}
\bpages{563--582}.
\end{barticle}
\endbibitem

\bibitem{bronski2}
\begin{barticle}[author]
\bauthor{\bsnm{Bronski},~\bfnm{Jared~C}\binits{J.~C.}}
(\byear{2003}).
\btitle{Small ball constants and tight eigenvalue asymptotics for fractional
  Brownian motions}.
\bjournal{Journal of Theoretical Probability}
\bvolume{16}
\bpages{87--100}.
\end{barticle}
\endbibitem

\bibitem{chigansky2018exact}
\begin{barticle}[author]
\bauthor{\bsnm{Chigansky},~\bfnm{Pavel}\binits{P.}} \AND
  \bauthor{\bsnm{Kleptsyna},~\bfnm{Marina}\binits{M.}}
(\byear{2018}).
\btitle{Exact asymptotics in eigenproblems for fractional Brownian covariance
  operators}.
\bjournal{Stochastic Processes and their Applications}
\bvolume{128}
\bpages{2007--2059}.
\end{barticle}
\endbibitem

\bibitem{OU}
\begin{barticle}[author]
\bauthor{\bsnm{Corlay},~\bfnm{Sylvain}\binits{S.}} \betal{et~al.}
(\byear{2010}).
\btitle{Functional quantization-based stratified sampling methods}.
\bjournal{arXiv preprint arXiv:1008.4441}.
\end{barticle}
\endbibitem

\bibitem{bridge}
\begin{barticle}[author]
\bauthor{\bsnm{Deheuvels},~\bfnm{Paul}\binits{P.}}
(\byear{2007}).
\btitle{A Karhunen--Lo{\`e}ve expansion for a mean-centered Brownian bridge}.
\bjournal{Statistics \& probability letters}
\bvolume{77}
\bpages{1190--1200}.
\end{barticle}
\endbibitem

\bibitem{embed}
\begin{barticle}[author]
\bauthor{\bsnm{Dietrich},~\bfnm{CR}\binits{C.}} \AND
  \bauthor{\bsnm{Newsam},~\bfnm{Garry~Neil}\binits{G.~N.}}
(\byear{1997}).
\btitle{Fast and exact simulation of stationary Gaussian processes through
  circulant embedding of the covariance matrix}.
\bjournal{SIAM Journal on Scientific Computing}
\bvolume{18}
\bpages{1088--1107}.
\end{barticle}
\endbibitem

\bibitem{ZT}
\begin{barticle}[author]
\bauthor{\bsnm{Dzhaparidze},~\bfnm{Kacha}\binits{K.}} \AND
  \bauthor{\bsnm{Van~Zanten},~\bfnm{Harry}\binits{H.}}
(\byear{2004}).
\btitle{A series expansion of fractional Brownian motion}.
\bjournal{Probability theory and related fields}
\bvolume{130}
\bpages{39--55}.
\end{barticle}
\endbibitem

\bibitem{GL}
\begin{barticle}[author]
\bauthor{\bsnm{Igl{\'o}i},~\bfnm{Endre}\binits{E.}}
(\byear{2005}).
\btitle{A rate-optimal trigonometric series expansion of the fractional
  Brownian motion}.
\bjournal{Electron. J. Probab}
\bvolume{10}
\bpages{1381--1397}.
\end{barticle}
\endbibitem

\bibitem{ITO}
\begin{barticle}[author]
\bauthor{\bsnm{It{\^o}},~\bfnm{Kiyosi}\binits{K.}},
  \bauthor{\bsnm{Nisio},~\bfnm{Makiko}\binits{M.}} \betal{et~al.}
(\byear{1968}).
\btitle{On the convergence of sums of independent Banach space valued random
  variables}.
\bjournal{Osaka Journal of Mathematics}
\bvolume{5}
\bpages{35--48}.
\end{barticle}
\endbibitem

\bibitem{JG}
\begin{barticle}[author]
\bauthor{\bsnm{Junglen},~\bfnm{Stefan}\binits{S.}} \AND
  \bauthor{\bsnm{Luschgy},~\bfnm{Harald}\binits{H.}}
(\byear{2010}).
\btitle{A constructive sharp approach to functional quantization of stochastic
  processes}.
\bjournal{Journal of Applied Mathematics}
\bvolume{2010}.
\end{barticle}
\endbibitem

\bibitem{KH}
\begin{barticle}[author]
\bauthor{\bsnm{K{\"u}hn},~\bfnm{Thomas}\binits{T.}} \AND
  \bauthor{\bsnm{Linde},~\bfnm{Werner}\binits{W.}}
(\byear{2002}).
\btitle{Optimal series representation of fractional Brownian sheets}.
\bjournal{Bernoulli}
\bpages{669--696}.
\end{barticle}
\endbibitem

\bibitem{leon2020stratonovich}
\begin{barticle}[author]
\bauthor{\bsnm{Le{\'o}n},~\bfnm{Jorge~A}\binits{J.~A.}} \betal{et~al.}
(\byear{2020}).
\btitle{Stratonovich type integration with respect to fractional Brownian
  motion with Hurst parameter less than $1/2$}.
\bjournal{Bernoulli}
\bvolume{26}
\bpages{2436--2462}.
\end{barticle}
\endbibitem

\bibitem{li2011approximations}
\begin{barticle}[author]
\bauthor{\bsnm{Li},~\bfnm{Yuqiang}\binits{Y.}},
  \bauthor{\bsnm{Dai},~\bfnm{Hongshuai}\binits{H.}} \betal{et~al.}
(\byear{2011}).
\btitle{Approximations of fractional Brownian motion}.
\bjournal{Bernoulli}
\bvolume{17}
\bpages{1195--1216}.
\end{barticle}
\endbibitem

\bibitem{LH}
\begin{barticle}[author]
\bauthor{\bsnm{Luschgy},~\bfnm{Harald}\binits{H.}} \AND
  \bauthor{\bsnm{Pag{\`e}s},~\bfnm{Gilles}\binits{G.}}
(\byear{2002}).
\btitle{Functional quantization of Gaussian processes}.
\bjournal{Journal of Functional Analysis}
\bvolume{196}
\bpages{486--531}.
\end{barticle}
\endbibitem

\bibitem{PG}
\begin{barticle}[author]
\bauthor{\bsnm{Luschgy},~\bfnm{Harald}\binits{H.}} \AND
  \bauthor{\bsnm{Pag{\`e}s},~\bfnm{Gilles}\binits{G.}}
(\byear{2009}).
\btitle{Expansions for Gaussian processes and Parseval frames}.
\bjournal{Electron. J. Probab}
\bvolume{14}
\bpages{1198--1221}.
\end{barticle}
\endbibitem

\bibitem{brunski4}
\begin{barticle}[author]
\bauthor{\bsnm{Luschgy},~\bfnm{Harald}\binits{H.}},
  \bauthor{\bsnm{Pag{\`e}s},~\bfnm{Gilles}\binits{G.}} \betal{et~al.}
(\byear{2004}).
\btitle{Sharp asymptotics of the functional quantization problem for Gaussian
  processes}.
\bjournal{The Annals of Probability}
\bvolume{32}
\bpages{1574--1599}.
\end{barticle}
\endbibitem

\bibitem{LAST}
\begin{barticle}[author]
\bauthor{\bsnm{Luschgy},~\bfnm{Harald}\binits{H.}},
  \bauthor{\bsnm{Pag{\`e}s},~\bfnm{Gilles}\binits{G.}} \betal{et~al.}
(\byear{2007}).
\btitle{High-resolution product quantization for Gaussian processes under
  sup-norm distortion}.
\bjournal{Bernoulli}
\bvolume{13}
\bpages{653--671}.
\end{barticle}
\endbibitem

\bibitem{bronski3}
\begin{barticle}[author]
\bauthor{\bsnm{Nazarov},~\bfnm{Aleksandr~Il'ich}\binits{A.~I.}} \AND
  \bauthor{\bsnm{Nikitin},~\bfnm{Ya~Yu}\binits{Y.~Y.}}
(\byear{2005}).
\btitle{Logarithmic L2-small ball asymptotics for some fractional Gaussian
  processes}.
\bjournal{Theory of Probability \& Its Applications}
\bvolume{49}
\bpages{645--658}.
\end{barticle}
\endbibitem

\bibitem{carlo}
\begin{barticle}[author]
\bauthor{\bsnm{Printems},~\bfnm{Jacques}\binits{J.}} \betal{et~al.}
(\byear{2005}).
\btitle{Functional quantization for numerics with an application to option
  pricing}.
\bjournal{Monte Carlo Methods and Applications mcma}
\bvolume{11}
\bpages{407--446}.
\end{barticle}
\endbibitem

\bibitem{wang1985convex}
\begin{barticle}[author]
\bauthor{\bsnm{Wang},~\bfnm{Hann~Tzong}\binits{H.~T.}}
(\byear{1985}).
\btitle{Convex functions and Fourier coefficients}.
\bjournal{Proceedings of the American Mathematical Society}
\bvolume{94}
\bpages{641--646}.
\end{barticle}
\endbibitem

\end{thebibliography}

\end{document}